\title{Independent sets in graphs with given\\ minimum degree}

\documentclass[11pt]{article}
\usepackage{amsfonts, amsmath, amsthm, amstext, amssymb, amscd, mathrsfs, verbatim, graphicx, color}
\newtheorem{theorem}{Theorem}
\newtheorem{lemma}[theorem]{Lemma}
\newtheorem{corollary}[theorem]{Corollary}

\newcommand{\cI}{\mathcal I}
\newcommand{\cG}{\mathcal G}

\newcommand{\E}{\mathbb E}
\newcommand{\pr}{\mathbb P}

\newcommand{\cg}{{\cal G}}
  \newcommand{\kk}{ K^*_{d,n-d} }
  
  \usepackage{amsmath}
  \usepackage{amssymb}
  \usepackage{latexsym}

  {\hspace*{\fill}$\Box$}

\date{September 26, 2012}

\begin{document}

\author{Hiu-Fai Law \and Colin McDiarmid}
\maketitle

\begin{abstract}
  We consider numbers and sizes of independent sets in graphs with minimum degree at least $d$,
  when the number $n$ of vertices is large. In particular we investigate which of these
  graphs yield the maximum numbers of independent sets of different sizes,
  and which yield the largest random independent sets.
  We establish a strengthened form of a conjecture of Galvin concerning the first of these topics.
\end{abstract}

  Given a graph $G$, let $\cI(G)$ be the set of independent sets and let $i(G)=|\cI(G)|$; and
  for $k\geq 0$ let $\cI_k(G)$ be the set of independent sets of order $k$ and let $i_k(G)= |\cI_k(G)|$.
  Thus $i(G) = \sum_{k\geq 0} i_k(G)$.
  
  There are many extremal results on $i(G)$ and $i_k(G)$, where $G$ ranges over a certain family of graphs, for example, trees or regular graphs  
  (see \cite{CGT09}-\!\cite{Gal11}, \cite{Kah01}-\!\cite{PT82},\cite{Zha10}).
  Here we investigate graphs with a given lower bound on their vertex degrees.
  For $d\geq 0$, let $\cG_n(d)$ be the set of graphs of order $n$ with minimum degree at least $d$.
  (Always $n,k$ and $d$ will be integers.)
  We are interested in which of these graphs yield the maximum numbers of independent sets of
  different sizes, and which yield the largest random independent sets. 
  Let us discuss numbers first.

  Recall that the \emph{independence number} $\alpha(G)$ is the maximum size of an independent set.
  Clearly $\alpha(G) \leq n-d$ for each $G \in \cG_n(d)$. 
  Recently, Galvin~\cite{Gal11} proved that, 
  for $n$ suitably larger than $d$,
  $i(G)< i(K_{d,n-d})$ for any $G\in\mathcal G_n(d)$ that is not (isomorphic to) $K_{d,n-d}$.
  Moreover, 
  he conjectured essentially 
  that for any $d\geq 1$, there exist integers
  $N(d)$ and $C(d)$ such that for each $n \geq N(d)$, $K_{d,n-d}$ maximizes
  $i_k$ over all graphs in $\mathcal G_n(d)$ for each $k$ satisfying $C(d) \leq k \leq n-d$;
  and he proved  
  such a result in the case when $d=1$.
 
  We shall see that this conjecture holds even if $d$ is allowed to grow slowly, and further we can take $C(d)=3$.
  Observe that we need $C(d) \geq 3$. For, each $n$-vertex graph has $i_0(G)=1$ and $i_1(G)=n$.
  Also $i_2(G)= \binom{n}{2} - e(G)$, where $e(G)$ is the number of edges, and graphs $G \in \cg_n(d)$
  can have $i_2(G)>i_2(K_{d,n-d})$.
  (For example, if $d$ is fixed and $n$ is large and even, $K_{d,n-d}$ has $d(n-d) \sim dn$ edges, whereas
  a $d$-regular graph has $dn/2$ edges.)  
  We shall show:
\begin{theorem} \label{thm.indepksets}
  Let $1\leq d=d(n)=o(n^{1/3})$. 
  Then for all sufficiently large~$n$, 
  for each graph $G \in \cg_n(d)$ and each $k \geq 3$ we have
  $i_k(G)\leq i_k(K_{d,n-d})$;
  and if $G$ is not $K_{d,n-d}$ 
  then $i_2(G)+i_{4}(G) < i_2(K_{d,n-d}) + i_{4}(K_{d,n-d})$,
  and so $i(G)<i(K_{d,n-d})$.
\end{theorem}  

  A graph $G \in \cg_n(d)$ with $\alpha(G)=n-d$ has the form $G = H+I_{n-d}$
  for a graph $H$ of order $d$ and the empty graph $I_{n-d}$ on $n-d$ vertices. 
  (Recall that for graphs $G, G'$ with disjoint vertex sets, the sum $G+G'$ denotes the graph obtained by adding
  all edges between them.)  Let $K^*_{a,b}$ denote the graph $K_a + I_b$.

  Denote by $X(G)$ the size of an independent set chosen uniformly at random from $\mathcal I(G)$.
  Recall that $X$ is \emph{stochastically dominated} by $Y$,
  denoted by $X\leq_s Y$, if $\mathbb P(X\leq t)\geq \mathbb P(Y\leq t)$ for each $t$.

  If $G \in \cg_n(d)$ satisfies $\alpha(G)=n-d$ and $G$ is not $\kk$, 
  then $G$
  is (isomorphic to) a proper subgraph of $K^*_{d,n-d}$, and so
  $i(G)> i(K^*_{d,n-d})$; and it follows that
  $\pr(X(G) \leq t) < \pr(X(\kk) \leq t)$ for $t=0$ and $t=1$.
  Hence it is {\em not} the case that $X(G) \leq_s X(\kk)$.
  Nevertheless, our second theorem shows that,
  if we ignore independent sets of size at most 1, then of all graphs in $\cg_n(d)$,
  the graph $\kk$ is the unique graph yielding the largest random independent sets. 
  
\begin{theorem} \label{T:SDmax}
  Let $1\leq d=d(n)=o(n^{1/3})$. 
  Then for all sufficiently large~$n$, for each graph $G \in\mathcal G_n(d)$ other than $\kk$, we have 
\[ \pr(X(G)\geq t) < \pr(X(K^*_{d,n-d})\geq t) \;\; \mbox{ for each } t=3,\ldots,n-d, \]
  and if $\alpha(G) < n-d$ then this inequality holds also for $t=1$ and $2$.
\end{theorem}
\noindent
  This yields directly:  
\begin{corollary} \label{cor.1}
  If $d$ is as above, then for all sufficiently large~$n$, for each graph $G \in\mathcal G_n(d)$
\begin{equation} \label{eqn.cor2}
  X(G)\leq_s \max\{ 2, X(K^*_{d,n-d})\},
\end{equation}
  and
\begin{equation} \label{eqn.cor1}
  \mbox{ if } \alpha(G) < n-d \;\; \mbox{ then } \; X(G) \leq_s X(\kk).
\end{equation}
\end{corollary}

\noindent
  Also, since $\mathbb E(X) = \sum_{t\geq 1} \mathbb P(X\geq t)$, we may obtain almost directly:
\begin{corollary} \label{cor.2}
  If $1\leq d=d(n)=o(n^{1/3})$, then for all sufficiently large~$n$,
  for each graph $G \in\mathcal G_n(d)$ other than $\kk$, we have 
\[ \mathbb E(X(G)) < \mathbb E(X(K^*_{d,n-d})) < (n-d)/2. \]
\end{corollary}
\noindent

  In order to prove these results,  it turns out that the `growth rates' $\alpha_k$
  of the numbers of independent sets are crucial quantities. For a graph $G$ and positive integer $k \leq \alpha(G)$,
  let $\alpha_k(G) := \frac{i_k(G)}{i_{k-1}(G)}$.
  Thus $\alpha_k(G)$ is $1/k$ times the average number of extensions of an independent $(k-1)$-set
  to an independent $k$-set in $G$;
  or (roughly) the `average number of extensions per vertex' at size~$k$. 

  To prove Theorem~\ref{thm.indepksets} we use two lemmas, one on growth rates $\alpha_k(G)$ and
  one on the `base case' $i_3(G)$.  To prove Theorem~\ref{T:SDmax} we need one further lemma,
  a general result on growth rates and stochastic domination.

  We adopt the following notations. For a graph $G$ and integer $d$  
  let $A=A(G,d)=\{ v\in V(G): \deg(v)>d\}$ and $B=V(G)\setminus A$; and let $a=|A|$, $b=|B|$.
  Also recall the standard notation that, if $U$ is a set of vertices in~$G$, then the neighbourhood $\Gamma(U)$
  is the set of neighbours of vertices in $U$,
  and the closed neighbourhood $\Gamma[U]$ is $\Gamma(U) \cup U$.
\begin{lemma}
\label{L:avgext}
  (a) For each $1\leq d<n$ 
  and $G \in \cg_n(d)$, we have
   $\alpha_k(G)\leq \alpha_k(\kk)$ for each $3\leq k\leq \alpha(G)$.\\
  (b) 
  Let $1\leq d=d(n)=o(n^{1/3})$.  
  Then for all sufficiently large $n$, for each $G,K\in \mathcal G_n(d)$ with $\alpha(G)<n-d=\alpha(K)$,
  we have $\alpha_k(G) < \alpha_k(K)$
  for each $4\leq k\leq \alpha(G)$.
\end{lemma}
\begin{proof}
  Let $3\leq k\leq \alpha(G)$. 
  Since each vertex degree in $G$ is at least $d$,
  each $I\in \mathcal I_{k-1}(G)$ can be extended to at most $n-d-k+1$ independent $k$-sets.
  Call $I$ \emph{good} if this upper bound is attained,
  and otherwise call $I$ \emph{bad}.
  Note that $I$ is good if and only if $|\Gamma(I)|=d$,
  if and only if each vertex in $I$ has the same set of $d$ neighbours. 
  Also, each $I$ is good if $G$ is $\kk$.
   
  Since each independent $k$-set contains exactly $k$ independent $(k-1)$-sets,
  we have $i_{k-1}(G)(n-d-k+1)\geq ki_k(G)$.
  Hence, $\alpha_k(G) \leq \frac{n-d-k+1}{k}$.
  But $\alpha_k(\kk)=\frac{n-d-k+1}{k}$ for $k= 3,\ldots,n-d$. This establishes part (a).
  
  Now we prove part (b).  Let $4 \leq k \leq \alpha(G)$.  Suppose first that $k \geq d+2$.
  Let $J$ be an independent set in $G$ of size $\alpha(G) \leq n-d-1$.
  Let $W$ be a set of $d+1$ vertices outside $J$, and note that each vertex in $W$ has at least one neighbour in $J$.
  Since $k-1 \geq d+1$ we may pick a $(k-1)$-subset $I$ of $J$ with $\Gamma(I) \supseteq W$,
  and so $I$ is bad.  Now, since there is a bad independent $(k-1)$-set, $\alpha_k(G)< \frac{n-d-k+1}{k}$.
  Further, $\alpha_k(K)=\frac{n-d-k+1}{k}$ for each $k= d+2,\ldots,n-d$, so this case is done; and
  so to prove part (b) we may assume that $4 \leq k\leq d+1$. 

  Assume also that $n >2d$ (as we may).  Write $K=H+ I_{n-d}$ for some graph $H$ of order $d$.
  Then $\alpha_k(K) = \frac{\binom{n-d}{k}+i_k(H)}{\binom{n-d}{k-1} + i_{k-1}(H)}$.
  Since $i_{k-1}(H)\leq \binom{d}{k-1}$, for each $k \leq d+1$,
\begin{equation} \label{eqn.alphak}
  \alpha_k(K) \geq \frac{\binom{n-d}{k}}{\binom{n-d}{k-1}+\binom{d}{k-1}}
   >   \frac{n-d-k+1}{k} \left( 1- \frac{\binom{d}{k-1}}{\binom{n-d}{k-1}}\right).
\end{equation}

  Let $p$ and $q$ denote the numbers of good and bad sets in $\mathcal I_{k-1}(G)$ respectively,
  so $p+q=i_{k-1}(G)$. Then
\[ k i_k(G) \leq p (n-d-k+1) + q(n-d-k) = (p+q) (n-d-k+1) -q, \]
  so
\begin{equation} \label{eqn.alphak2}
  \alpha_k(G) \leq \frac{n-d-k+1}{k} - \frac{q}{k(p+q)}.
\end{equation}
  Assume for a contradiction that $\alpha_k(G) \geq \alpha_k(K)$. 
  Then it follows using~(\ref{eqn.alphak}) and~(\ref{eqn.alphak2}) that
\begin{equation}
\label{E:badprop}
\frac{q}{p+q} \leq (n-d-k+1) \binom{d}{k-1}/\binom{n-d}{k-1}< \frac{d^{k-1}}{(n-d-k+1)^{k-2}}.
\end{equation}
  Observe that, since $k \geq 4$, the final bound above is $O(d^3 n^{-2})=o(n^{-1})$.
  Thus certainly $p>0$.
  \smallskip
  
  \noindent
{\bf Claim:}
  For each good independent $(k-1)$-set $I$ in $G$ there is a vertex $w \not\in I \cup \Gamma(I)$ such that
  $\Gamma(w) \neq \Gamma(I)$.
  \medskip
  
  We will prove the claim later: suppose for now that it holds.
  Then from each good independent $(k-1)$-set $I$ we may construct a bad independent $(k-1)$-set $I'$
  by deleting a vertex $u$ from $I$ and adding a vertex $w$ as in the claim.
  This gives at least $p(k-1) \geq 3p$ constructions.  Also, in each bad independent $(k-1)$-set $I'$ which has been
  constructed, we can identify the vertex $w$ added
  (since the other $k-2 \geq 2$ vertices all have the same neighbourhood).
  Thus each bad independent $(k-1)$-set $I'$ is constructed at most $n-k+1 \leq n-3$ times.  Hence
\[ q \geq 3p/(n-3) > p/(n-1) \]
  and so $q/(p+q) > 1/n$, which contradicts \eqref{E:badprop} (for $n$ sufficiently large, since $k \geq 4$).
  
  It remains to prove the claim.
  Recall that $B=\{v \in V(G): \deg(v)=d\}$.
  Let $I$ be a good independent $(k-1)$-set.
  Note that $I \subseteq B$ and $|\Gamma(I)|=d$.
  If $|A|=a \geq d+1$ then for $w$ we may pick any vertex in $A \setminus \Gamma(I)$.  So we may assume that $a \leq d$.
  
  Let $B_1=\{v\in B: \Gamma(v)\cap B\neq \emptyset\}$ and $B_2 = B \setminus B_1$.
  Since $\alpha(G)<n-d \leq |B|$ we have $E(B) \neq \emptyset$ and so $B_1 \neq \emptyset$.
  Either $I \subseteq B_1$ or $I \subseteq B_2$, since
  each vertex in $I$ has the same set of $d$ neighbours.
  If $I \subseteq B_1$ then $I \subseteq \Gamma(v)$ for some $v \in B_1$,
  and so for $w$ we may pick any vertex not in $\Gamma(I) \cup \Gamma(v)$
  (at least $n-2d \geq 1$ choices).
  If $I \subseteq B_2$ then for $w$ we may pick any vertex in $B_1$.
  This completes the proof of the claim, and we are done.
\end{proof}

  The previous lemma concerns ratios; the next considers the base case.
  Of graphs in $\cg_n(d)$, clearly a $d$-regular graph has the most independent $2$-sets:
  we look at the number $i_3$ of independent $3$-sets. We first give a formula for $i_3(G)$ for any graph $G$.
  Let $t_i$ be the number of induced subgraphs of $G$ on three vertices with $i$ edges.  Then
\begin{eqnarray*}
\binom{n}{3} & = &t_0 +t_1 +t_2 +t_3,\\
  e(G) (n-2) &= &t_1 +2t_2 +3t_3, \\
\sum_{v_i\in V(G)} \binom{\deg(v_i)}{2} & = &t_2 + 3t_3.
\end{eqnarray*}
  Hence,
\begin{equation} \label{E:i3}
 i_3(G) = \binom{n}{3} - e(G)(n-2) + \sum_{v_i\in V(G)} \binom{\deg(v_i)}{2} - t(G),
\end{equation}
  where $t(G)=t_3$ is the number of triangles.  For example, if $G$ is a $d$-regular graph then
\begin{eqnarray*}
i_3(G) & = & \binom{n}{3} - \frac12 dn(n-2) +n \binom{d}{2} -t(G) \\
&=& \binom{n-d}{3} - \frac12 dn + \frac16 d(d^2+3d+2) - t(G).
\end{eqnarray*}

\begin{lemma}
\label{L:triangle}
  Let $1\leq d=d(n)=o(n^{1/3})$. 
  For all sufficiently large~$n$, if $G, K\in\mathcal G_n(d)$
  are such that $\alpha(G)<n-d=\alpha(K)$, then $i_3(G) \leq i_3(K)- n/2 +1$.
\end{lemma}
\begin{proof}
  Our proof relies on~(\ref{E:i3}).  Consider $G\in\mathcal G_n(d)$ with $\alpha(G)<n-d$.
  We first show that we may assume without loss of generality
  that the set $A$ of vertices of degree $>d$ is a non-empty
  independent set, and then that it suffices to prove~(\ref{claim.hd}) below;
  then we prove~(\ref{claim.hd}) by considering four cases for $a=|A|$.
  
  Suppose that $G$ is $d$-regular. Then by the above we have
\[ i_3(G) \leq \binom{n-d}{3} - \frac12 dn + \frac16 d(d^2 +3d+2).\]
  But $i_3(K) \geq \binom{n-d}{3}$.  Thus, if $d=1$ then
\[ i_3(G) \leq \binom{n-d}{3} -n/2+1 \leq i_3(K) -n/2 +1;\]
 and if $d \geq 2$ then
\[ i_3(G) \leq \binom{n-d}{3} -n +O(d^3) \leq i_3(K) -n/2\]
 for $n$ sufficiently large.
 Hence we may assume that $G$ is not regular, and so $A$ is non-empty.

  Now repeatedly delete edges between vertices of degree $>d$, 
  as long as $G$ keeps satisfying $\alpha(G)<n-d$.
  We end up with some graph $G' \in\mathcal G_n(d)$ with $\alpha(G')<n-d$.  Suppose that there is an edge $uv\in E'(A')$ after this step (we use $E'$ and $A'$ to refer to $G'$). Then there exists an $(n-d)$-set $I$ such that $E'(I)=\{uv\}$. Let $J=V(G')\setminus I$, so $|J|=d$. Since $\deg_{G'}(u), \deg_{G'}(v)>d$ and $\deg_{G'}(w) \geq d$ for each other vertex $w \in I$, every possible edge between $I$ and $J$ is present in $G'$.
Therefore, since there are $(n-d-2)$ 3-subsets of $I$ containing $u$ and $v$,
\begin{eqnarray*}
i_3(G) \leq i_3(G')  
       & \leq    & \binom{n-d}{3} - (n-d-2) + \binom{d}{3} < \binom{n-d}{3} - \frac{n}{2}
\end{eqnarray*}
for large $n$, since $d=o(n^{1/3})$. Hence, we may assume that $A$ is independent. 

  For each $v_i\in A$, let $r_i=\deg(v_i)$. 
  Observe that $2e(G)= \sum_i r_i \ + (n-a)d$.
  Thus, from~\eqref{E:i3},  
\begin{eqnarray*}
  && 2 i_3(G) - 2\binom{n}{3}\\
  &=& -[\sum_{i=1}^{a} r_i \ + (n-a)d ] (n-2) + \sum_{i=1}^{a} r_i(r_i -1) + (n-a)d(d-1)\!-\! 2 t(G)\\
  &=& \sum_{i=1}^{a} r_i(r_i -n +1) -(n-a)d(n-d-1) - 2 t(G)\\
  &=& - dn(n-d-1) + h_d(G),
\end{eqnarray*}
where
\[h_d(G) = \sum_{i=1}^a r_i (r_i-n+1) + ad(n-1-d) - 2t(G). \]
Thus
\begin{eqnarray*}
  i_3(G) & = & \binom{n}{3} - \frac12 dn(n-d-1) + \frac12 h_d(G)\\
  & =& \binom{n-d}{3} - \frac12 dn + \frac16 d (d^2+3d+2) + \frac12 h_d(G).
\end{eqnarray*}
  Observe that here only $\frac12 h_d(G)$ varies with $G\in {\mathcal G}_n(d)$.
Since $i_3(K)\geq \binom{n-d}{3}$, by the last equality
\[ h_d(K) \geq dn - \frac13 d (d^2+3d+2) = (1+o(1))\ dn.\]
Thus it suffices to show that
\begin{equation} \label{claim.hd}
  h_d(G) \leq (d-1)n +O(d^2)
\end{equation}
  and the remainder of the proof is devoted to establishing this result.

  Recall that we are assuming that in $G$ the set $A$ of vertices of degree $>d$ is independent.
  Thus $d+1 \leq r_i \leq n-a$ for each $i=1,\ldots,a$.
  Consider the function $g(x)=x(x-n+1)= -x(n-1-x)$ for real $x$.  This is decreasing for $x<(n-1)/2$ and increasing for $x>(n-1)/2$.
  We now break the proof of~(\ref{claim.hd}) into four cases: $a \geq d+2$, $a=d+1$, $a=d$, and $1 \leq a \leq d-1$. 
\smallskip

  Suppose that $a \geq d+2$. Then each $d+1 \leq r_i \leq n-d-2$, so $g(r_i) \leq (d+1)(d+2-n)$.
Hence,
\begin{eqnarray}
  h_d(G) & \leq & a (d+1)(d+2-n) + ad(n-1-d) \nonumber\\
         &  =   &  a(-n+2d+2) \label{E:a_large}
\end{eqnarray}
 and so~(\ref{claim.hd}) holds.

Suppose that $a=d+1$.  Then $d+1 \leq r_i \leq n-d-1$ for each $i$, and $\sum_{i=1}^a r_i \leq d(n-a) = d(n-d-1)$.
Thus at most $d-1$ of the $r_i$ are equal to $n-d-1$, and so
\[\sum_{i=1}^a g(r_i) \leq -(d-1) d (n-d-1)  - 2 (d+1)(n-d-2)= -n(d^2+d+2) +O(d^3).\]
Hence $h_d(G) \leq -2n + O(d^3)$, and so~(\ref{claim.hd}) holds.

  Suppose that $a=d$. Since $\alpha(G)<n-d$, $e(B)>0$. It follows that $\sum_{i=1}^a r_i \leq d(n-d)-2 \leq d(n-d)-1$.
  Hence not all $d$ of the $r_i$ are equal to $n-d$, and so
\begin{eqnarray*}
h_d(G) & \leq & (d-1)(n-d)(1-d) + (n-d-1)(-d) + d^2(n-1-d)\\
   &=& (d-1)n +O(d^2)
\end{eqnarray*}
  as required. 
\smallskip

  Finally, suppose that $1 \leq a \leq d-1$. 
  Consider $v_i \in A$. Suppose that $r_i\geq n-d-1$. 
  Then the edge-boundary of $\Gamma(v_i)$ has size at most
  $r_i a +(n-a-r_i)d \leq r_ia + d(d+1-a)$, and so $2e(\Gamma(v_i)) \geq r_i(d-a) - d(d+1-a)$. 
  Hence, twice the number of triangles containing $v_i$ is at least $r_i(d-a)-d(d+1-a)$.
  Also, using first that $r_i \leq n-a$ and then that $r_i \geq n-d-1$ we have
\[ g(r_i)-r_i(d-a) = r_i(r_i-n+1-d+a) \leq r_i(1-d) \leq (n-d-1) (1-d).\]
  On the other hand, if $r_i\leq n-d-2$, then $g(r_i) \leq (d+1)(d-n+2)$. Let $l=|\{i: r_i\geq n-d-1\}|$. 
  Then $ h_d(G)$ is at most
  \begin{eqnarray*}
           & & \sum_{i: r_i\geq n\!-\!d\!-1}\!\! [g(r_i) - r_i(d\!-\!a) + d(d\!+\!1\!-\!a)]
       + \sum_{i: r_i\leq n\!-\!d\!-\!2} \!\!g(r_i) + ad(n\!-\!1\!-\!d)\\ 
      & \leq & l [(n\!-\!d\!-\!1)(1\!-\!d) + d(d\!+\!1\!-\!a)]\! +\! (a\!-\!l) (d\!+\!1)(d\!-\!n\!+\!2)\! +\! ad(n\!-\!1\!-\!d)\\
      & \leq & an + O(d^2) \; \leq \; (d-1)n + O(d^2)
  \end{eqnarray*}
  as required.  
  \end{proof}

  With the last two lemmas, we may now prove Theorem~\ref{thm.indepksets}, establishing a stronger version of the
  conjecture of Galvin \cite{Gal11} mentioned earlier.
\smallskip

\begin{proof}[Proof of Theorem~\ref{thm.indepksets}]
  If $\alpha(G)= n-d$ then $G$ is (isomorphic to) a 
  supergraph of $K_{d,n-d}$ and the result is trivial:
  so we may assume that $\alpha(G)<n-d$.  Let us also assume that $n$ is large.  Let $K \in \cg_n(d)$
  with $\alpha(K)=n-d$.
  
  Since $i_3(G)\leq i_3(K)-n/2 +1$ by Lemma~\ref{L:triangle} (b), by Lemma \ref{L:avgext} we have
  $i_k(G)< i_k(K)$ for all $k\geq 3$. In fact, $i_4(G)<i_4(K) - \Omega(n^2)$ since $\alpha_4(K)=\Omega(n)$.
  On the other hand, $e(G)\geq dn/2$, so that $i_2(G)-i_2(K) \leq dn/2$.
  Thus $i_2(G)+i_4(G) < i_2(K) + i_4(K)$, and we are done.
  \end{proof}
\medskip  

  To prove Theorem~\ref{T:SDmax}, as well as the two corollaries,
  we need one further lemma, which is
  a general result on growth rates and stochastic domination,
  adapted from Lemma 2.4 of~\cite{msw05}.
  Given a finite sequence of positive real numbers $x = (x_0,x_1,\ldots,x_s)$,
  let $S(x) = \sum_{k\geq 0} x_k$. 
  Define a random variable $X=X(x)$ by $\mathbb P(X=k) = x_k/S(x)$.

\begin{lemma} \label{L:SD}
  Let $x_0, y_0>0$, let $1 \leq a \leq b$ be integers, and let $\alpha_1, \ldots, \alpha_a > 0$ and
  $\beta_1, \ldots, \beta_b > 0$.
  For $i=1,\ldots, a$, let $x_i = x_0 \prod_{0<j\leq i} \alpha_j$; and for $i=1,\ldots, b,$
  let $y_i = y_0 \prod_{0<j\leq i}\beta_j$.
  Let $x = (x_0,x_1,\ldots,x_a)$ and $y = (y_0,y_1,\ldots,y_b)$, and
  denote $X(x)$ by $X$ and $X(y)$ by $Y$.
  If $\alpha_i\leq \beta_i$ for each $i=1,\ldots, a$, then $X\leq_s Y$.
  Further, if these conditions hold, and $(\alpha_1, \ldots, \alpha_a) \neq (\beta_1, \ldots, \beta_b)$,
  then
\[ \pr(X \geq t) < \pr(Y \geq t) \mbox{ for each } t=1,\ldots,b. \] 
\end{lemma}
\begin{proof}
  By replacing $y_a$ by $\sum_{j>a} y_j$, we may assume that $b=a$. It suffices to consider the case when
  $\alpha_i=\beta_i$ for all $i$ except $j_0$ where $\alpha_{j_0}<\beta_{j_0}$.
  Since $\mathbb P(X\leq a)=\mathbb P(Y\leq a)=1$, it suffices to prove
  $\mathbb P(X\leq t) > \mathbb P(Y\leq t)$ for $t=0,\ldots, a-1$.
  Note that we may rescale $x_i, y_i$'s without changing the distribution. 

  Suppose $t$ satisfies $0\leq t\leq j_0-1$. Rescale to $x_0=y_0=1$. Then $x_i=y_i$ for all
  $i\leq t$ and $S(x)<S(y)$. So
  $\pr(X\leq t)=\frac{\sum_{i\leq t} x_i}{S(x)} > \frac{\sum_{i\leq t} y_i}{S(y)}=\pr(Y\leq t)$. 

  For $t$ such that $j_0\leq t\leq a-1$, we rescale to $x_{j_0}=y_{j_0}$. Then
  $x_i=y_i$ for all $i=j_0, j_0+1,\ldots, a$ and $S(x)>S(y)$.
  Hence, $\pr(X>t) < \pr(Y>t)$ 
  and so $\pr(X\leq t) > \pr(Y\leq t)$.
\end{proof}

\begin{proof}[Proof of Theorem \ref{T:SDmax}]
  There are two cases, depending on whether $\alpha(G)<n-d$ or $\alpha(G)=n-d$.

  (a) Let $G \in \cg_n(d)$ with $\alpha(G)<n-d$. 
  For $k\geq 1$, let $\alpha^*_k$ denote $\alpha_k(\kk)$.
  Then $\alpha_1(G)=\alpha^*_1=n$.
  By Lemma \ref{L:avgext} (a),  
  $\alpha_k(G)\leq \alpha^*_k$ for $3 \leq k\leq \alpha(G)$. 
 
  If $\alpha_2(G) \leq \alpha^*_2$  
  then directly from 
  Lemma \ref{L:SD} we have $\pr(X(G) \geq t)< \pr(X(\kk)\geq t)$ for each $t=1,\ldots,n-d$, and we are done.
  So we may suppose that $\alpha_2(G) > \alpha^*_2$; that is $i_2(G) > i^*_2$,
  where $i_k^*$ denotes $i_k(\kk)$.
   
  Let $x$ be the $i_k$-vector for $G$ (up to $x_{n-d}$), let $z$ be the $i_k$-vector for $\kk$,
  and let $y$ agree with $x$ in the first three places, and agree with $z$ in the remaining places; that is,
\[ x=(x_0,x_1,\ldots,x_{n-d}) = (1,n,i_2(G),i_3(G),i_4(G),\ldots,i_{n-d}(G) ),\]
\[ y= (y_0,y_1,\ldots,y_{n-d})=(1,n,i_2(G), i^*_3,i^*_4,\ldots,i^*_{n-d})\]
  and
\[ z=(z_0,z_1,\ldots,z_{n-d})=(1,n,i^*_2, i^*_3,i^*_4,\ldots,i^*_{n-d}). \]
  Let $3 \leq t \leq n-d$. By Lemma~\ref{L:avgext} (b) with $K=\kk$, for each $4 \leq k \leq \alpha(G)$ we have
  $\frac{x_k}{x_{k-1}}\leq \frac{y_k}{y_{k-1}}$. 
  Moreover, by Lemma~\ref{L:triangle}, $i_3(G)<i^*_3$ so that $\frac{x_3}{x_2}<\frac{y_3}{y_2}$. 
  Then by Lemma~\ref{L:SD}, 
\[ \pr(X(G) \geq t) = \pr(X(x) \geq t) < \pr(X(y) \geq t).\] 
  Also
\[ \pr(X(y) \geq t) < \pr(X(z) \geq t) = \pr(X(\kk) \geq t)\]
  since $S(y)<S(z)$.   
  Hence
  $\pr(X(G) \geq t) < \pr(X(\kk) \geq t)$ as required.

  To complete the proof for this case, note that by Theorem~\ref{thm.indepksets}, 
  $i(G) < i(\kk)$, so that 
\[ \pr(X(G)\leq 0)=1/i(G) > 1/i(\kk)= \pr(X(\kk)\leq 0),\]
  and similarly
\[ \pr(X(G)\leq 1)=(1+n)/i(G) > (1+n)/i(\kk)= \pr(X(\kk)\leq 1). \]
  
  (b) It remains to consider the case when $\alpha(G)=n-d$ and $G$ is not $\kk$. 
  Then $G$ may be obtained from $\kk$ by deleting at least one edge from the $K_d$ part. Thus $i(G)>i(\kk)$;
  and the $i_k$-vector $x$ of $G$ may be obtained from the $i_k$-vector $z$ for $\kk$
  by adding positive integers to some entries amongst the first $d+1$ including adding at least 1 to $z_2$.
  It is immediate that $\pr(X(x) \geq t) < \pr(X(z)\geq t)$ for each $t=d+1,\ldots,n-d$.
  Let $2 \leq t \leq d-1$.  Then
\[ \pr(X(z) \leq t) = \frac{\sum_{i=0}^{t} z_i}{S(z)}.\]
  To obtain $\pr(X(x) \leq t)$ from the last ratio we add at least 1 to the numerator and at most $2^d$
  to the denominator. Thus the numerator increases by a factor $(1+\Omega(n^{-d}))$ and
  the denominator increases by a factor at most $(1+ 2^{-(n-2d)})$.
  So overall the ratio increases (for large $n$), that is $\pr(X(z) \leq t) < \pr(X(x) \leq t)$,
  as required.
\end{proof}

  We noted earlier that Corollary~\ref{cor.1} follows directly from Theorem~\ref{T:SDmax}, so it remains only to prove
Corollary~\ref{cor.2}.
\begin{proof}[Proof of Corollary~\ref{cor.2}]
If $\alpha(G) < n-d$, the result follows directly from~\eqref{eqn.cor1}. Suppose then that
$\alpha(G) = n-d$, and let $n$ be sufficiently large that $\E[X(\kk)] \geq d$.
Then the average size of the sets which are independent in $G$ but not in $\kk$ is at most $d \leq \E[X(\kk)]$,
and so $\E[X(G)] \leq \E[X(\kk)]$.
\end{proof}

  We remark that with an analogous method, a weighted version of the statements can be proved.
  Let $I(G, \lambda) = \sum_{k\geq 0} i_k(G)\lambda^k$ be the independent set polynomial of $G$ (\cite{GH83}, \cite{SS05}).
  Instead of a uniform sampling of independent sets of $\mathcal I(G)$, we fix $\lambda>0$ and
  pick a given independent $k$-set with probability $\lambda^k/ I(G,\lambda)$.
  Then under this sampling, the analogous versions of Theorem~\ref{T:SDmax} and its corollaries hold.


\end{document}